\documentclass[psamsfonts,fceqn,leqno]{amsart}
\usepackage{mathrsfs,latexsym,amsfonts,amssymb,curves,epic}
\usepackage{color}

\newtheorem{theorem}{Theorem}[section]
\newtheorem{corollary}[theorem]{Corollary}
\newtheorem{proposition}[theorem]{Proposition}
\newtheorem{lemma}[theorem]{Lemma}
\newtheorem{question}[theorem]{Question}
\newtheorem{problem}[theorem]{Problem}
\theoremstyle{definition}

\newtheorem{remark}[theorem]{Remark}
\newtheorem{example}[theorem]{Example}

\begin{document}
\title[Dense-separable groups and its applications in $d$-independence]
{Dense-separable groups and its applications in $d$-independence}

  \author{Fucai Lin}
  \address{(Fucai Lin): School of mathematics and statistics,
  Minnan Normal University, Zhangzhou 363000, P. R. China}
  \email{linfucai@mnnu.edu.cn; linfucai2008@aliyun.com}

  \author{Qiyun Wu}
  \address{(Qiyun Wu): School of mathematics and statistics,
  Minnan Normal University, Zhangzhou 363000, P. R. China}
\email{904993706@qq.com}

\author{Chuan Liu}
\address{(Chuan Liu)Department of Mathematics,
Ohio University Zanesville Campus, Zanesville, OH 43701, USA}
\email{liuc1@ohio.edu}

  \thanks{The first author is supported by the Key Program of the Natural Science Foundation of Fujian Province (No: 2020J02043), the NSFC (No. 11571158), the Institute of Meteorological Big Data-Digital Fujian and Fujian Key Laboratory of Data Science and Statistics.}

  %insert keywords
  %insert subject class

  %\date{\today}
  \begin{abstract}
A topological space is called {\it dense-separable} if each dense subset of its is separable. Therefore, each dense-separable space is separable. We establish some basic properties of dense-separable topological groups. We prove that each separable space with a countable tightness is dense-separable, and give a dense-separable topological group which is not hereditarily separable. We also prove that, for a Hausdorff locally compact group , it is locally dense-separable iff it is metrizable.

Moreover, we study dense-subgroup-separable topological groups. We prove that, for each compact torsion (or divisible, or torsion-free, or totally disconnected) abelian group, it is dense-subgroup-separable iff it is dense-separable iff it is metrizable.

Finally, we discuss some applications in $d$-independent topological groups and related structures. We prove that each regular dense-subgroup-separable abelian semitopological group with $r_{0}(G)\geq\mathfrak{c}$ is $d$-independent. We also prove that, for each regular dense-subgroup-separable bounded paratopological abelian group $G$ with $|G|>1$,
it is $d$-independent iff it is a nontrivial $M$-group iff each nontrivial primary component $G_{p}$ of $G$ is $d$-independent. Apply this result, we prove that a separable metrizable almost torsion-free paratopological abelian group $G$ with $|G|=\mathfrak{c}$ is $d$-independent. Further, we prove that each dense-subgroup-separable MAP abelian group with a nontrivial connected component is also $d$-independent.
  \end{abstract}

 \maketitle
\section{Introduction}
The dense subsets play an important role in the study of the theory of topological spaces, see \cite{AT2008, E1989}. The concept of the density of topological spaces is naturally studied. The density $d(X)$ of a space $X$, which is defined as the smallest cardinal number of the form $|A|$ for each dense subset $A$ of $X$, may determined the bounded of the weight of $X$, that is, $w(X)\leq 2^{d(X)}$; in particular, if $X$ is a separable regular space, then $w(X)\leq\mathfrak{c}$. Moreover, a compact dyadic space $X$ with $w(X)\leq 2^{\kappa}$ for some infinite cardinal $\kappa$ satisfies $d(X)\leq\kappa$; especially, if $w(X)\leq\mathfrak{c}$, then $X$ is separable. In 1975, R. Levy and R.H. McDowell \cite{LM1975} introduced the concept of dense-separable, that is, each dense subset of a space is separable; then, in 1976, J.H. Weston and J. Shilleto \cite{WS1976} studied the cardinality of dense subsets of topological spaces. In 1989, I. Juh\'{a}sz and S. Shelah \cite{JS1989} proved that the $\pi$-weight of a compact Hausdorff space $X$ is equal to the supremum of the density of all dense subsets of $X$. Recently, A. Dow and I. Juh\'{a}sz \cite{DJ2020} has proved that each dense subset of a compact space can be covered by countably many compact subsets iff it has a countable $\pi$-weight.

Moreover, it is an effective method to study topological or algebraic properties of dense subgroups of a (compact) topological group. In \cite{W1937}, A. Weil proved that each dense subgroup of compact subgroup is precompact, which provokes many topologist study the class of precompact topological groups, see \cite{AT2008}. In \cite{CD2014}, W. Comfort and D.N. Dikranjan also study dense subgroups of compact abelian groups, where they define the density nucleus, $den(G)$, of a topological group $G$ as the intersection of the family of dense subgroups of $G$; further, they studied the algebraic structures of maximally fragmentable compact topological groups. Then, E.M. Rodr\'{\i}guez and M. Tkachenko in \cite{RT2021} introduced the concept of $d$-independent in the class of topological groups, which is related to the subject of maximally fragmentable topological groups.

Until now, the study of dense subspaces or subgroups of topological groups is an interesting and meaningful topics, such as, the following Open Problems posed in \cite{AT2008}.

\begin{problem}\cite[Open Problem 4.1.1]{AT2008}
Let $G$ be a compact group. Is there a dense subspace $X$ of $G$ such that the tightness of $X$ is countable?
\end{problem}

\begin{remark}
It is well known that each compact group $G$ is dyadic, hence it follows from \cite[Theorem 10]{E1965} that $G$ is separable if $w(G)\leq\mathfrak{c}$, then there exists a countable dense subset $X$ of $G$ such that the tightness of $X$ is countable.
\end{remark}

\begin{problem}\cite[Open Problem 4.1.2]{AT2008}
Let $G$ be a compact group. Is there a dense sequential subspace of $G$?
\end{problem}

\begin{problem}\cite[Open Problem 4.1.3]{AT2008}
Let $G$ be a compact group. Is there a dense countably compact $\aleph_{0}$-monolithic subspace $X$
of $G$ such that the tightness of $X$ is countable?
\end{problem}

\begin{problem}\cite[Open Problem 4.1.4]{AT2008}
 Let $G$ be a compact group. Is there a dense subgroup of $G$ satisfying the conditions in one
of the above three problems?
\end{problem}

\begin{problem}\cite[Open Problem 1.4.3]{AT2008}
Does there exist an infinite (abelian, Boolean) topological group $G$ such that all dense subgroup of $G$ are connected?
\end{problem}

\begin{problem}\cite[Open Problem 1.4.4]{AT2008}
Does there exist an infinite (abelian) topological group $G$ such that all closed subgroup of $G$ are connected?
\end{problem}

This paper is organized as follows. In Section 2, we establish some basic properties of dense-separable spaces; in particular, dense-separable topological groups. We prove that each separable space with a countable tightness is dense-separable, and give a dense-separable topological group which is not hereditarily separable. We also prove that the product of a dense-separable regular space with a cosmic space is dense-separable. Further, we prove that, for a Hausdorff locally compact group, it is locally dense-separable iff it is metrizable.

Section 3 is dedicated to the study of dense-subgroup-separable topological groups. We prove that, for each compact torsion (or divisible, or torsion-free) abelian group, it is dense-subgroup-separable iff it is dense-separable iff it is metrizable. We also prove that, for each compact totally disconnected abelian group, it is dense-subgroup-separable iff it is dense-separable iff it is metrizable.

In Section 4, we discuss some applications in $d$-independent topological groups and related structures. We prove that, for a regular dense-subgroup-separable abelian semitopological group $G$, it is $d$-independent iff every nonempty open set in $G$ contains an independent subset of cardinality $\mathfrak{c}$. According to this result, we prove that each regular dense-subgroup-separable abelian semitopological group with $r_{0}(G)\geq\mathfrak{c}$ is $d$-independent. We also prove that, for each regular dense-subgroup-separable bounded paratopological abelian group $G$ with $|G|>1$,
it is $d$-independent iff it is a nontrivial $M$-group iff each nontrivial primary component $G_{p}$ of $G$ is $d$-independent. Apply this result, we prove that a separable metrizable almost torsion-free paratopological abelian group $G$ with $|G|=\mathfrak{c}$ is $d$-independent. Finally, we prove that each dense-subgroup-separable MAP abelian group $G$ with a nontrivial connected component is also $d$-independent.

 \maketitle
\subsection{Notation and terminology}
In this subsection, we introduce the
necessary notation and terminology which are used in
the paper. For undefined
  notation and terminology, the reader may refer to \cite{AT2008},
  \cite{E1989} and \cite{HM1998}.

All topological spaces are assumed to be Hausdorff. For a subset $A$ of a space $X$, the symbol $\overline{A}$ denotes the closure
of $A$ in $X$. Let $\mathbb{N}$, $\omega$, $\mathbb{R}$ and $\mathbb{T}$ denote the sets of all
positive integers, all non-negative integers, all real numbers and the unit circle group respectively. The cyclic group of order $n$ is $Z(n)$. The cardinality of a set $A$ will be denoted by $|A|$, and the symbol $\mathfrak{c}$ stands for the cardinality of $\mathbb{R}$.

Given a group $G$, the neutral element of $G$ is $e_{G}$ or simply $e$; moreover, if $G=\mathbb{T}$, then $e_{\mathbb{T}}$ is denoted by 1. Recall that an element $g$ of $G$ is {\it finite order} or {\it torsion} if $g^{n}=e$ (equivalently, $ng=e$, if $G$ is abelian) for some $n\in \mathbb{N}$, and the smallest $n\in \mathbb{N}$ for which $g^{n}=e$ is said to be the {\it order} of $g$. We say that $G$ is a {\it torsion group} if all elements of $G$ have finite orders, and say that $G$ is {\it torsion-free} if the group $G$ has no elements of finite order, except for $e$. A group $G$ is {\it divisible} if the equation $x^{n} =g$ (equivalently, $nx=g$, if $G$ is abelian) has a solution $x\in G$ for all $g\in G$ and $n\in\mathbb{N}$. Moreover, a torsion abelian group $G$ is said to be {\it bounded} if there is a positive integer $m$ such that $mg=e$, for each $g\in G$, and the minimum $m$ with this property is called the {\it period} of $G$. If a torsion group $G$ contains elements of arbitrary big order, then $G$ is said to be {\it unbounded}.

Suppose that $B$ is a nonempty subset of an abelian group $G$. For any pairwise distinct elements $b_{1}, \ldots, b_{m}\in B$ and any integers $n_{1}, \ldots, n_{m}$, if $n_{1}b_{1}+\cdots+n_{m}b_{m}=e$ implies $n_{i}b_{i}=e$ for every $i=1, \cdots, m$, then we say that $B$ is {\it linearly independent} or simply {\it independent subset} of $G$. Assume that $G$ is not torsion, then let $\mathcal{I}$ be the family of all independent subsets of $G$ consisting of elements of infinite order, which is partially ordered by inclusion. From Zorn's lemma, it follows that each element of $\mathcal{I}$ is contained in a maximal element of $\mathcal{I}$. If $A$ and $B$ are maximal elements of $\mathcal{I}$, then $|A|=|B|$ by \cite{R2012}. Then, we say that the cardinality of a maximal element of $\mathcal{I}$ is the {\it torsion-free rank} of $G$, which is denoted by $r_{0}(G)$.

Given an abelian group $G$, put $nG=\{ng: g\in G\}$ and $G[n]=\{g\in G: ng=e\}$ for each $n\in\mathbb{N}$. We say that an abelian group $G$ is {\it almost torsion-free} if each $G[n]$ is finite. Clearly, the elements of $G$ whose order is a power of a fixed prime $p$ form a subgroup $G_{p}$, then we say that $G_{p}$ is the {\it $p$-prime component of $G$}; in particular, if $G=G_{p}$ for some prime $p$, then $G$ is called a {\it $p$-group}.

For a semitopological group $G$ with $|G|=\kappa\geq\omega$, we say that $G$ is {\it maximally fragmentable} if it contains an independent family $\{D_{\alpha}: \alpha<\kappa\}$ of dense subgroups satisfying $D_{\alpha}\cap\langle\bigcup_{\beta\neq\alpha}D_{\beta}\rangle=\{e\}$ for each $\alpha<\kappa$. We say that a nontrivial semitopological group $G$ is {\it $d$-independent} if, for every subgroup $S$ of $G$ with $|S|<\mathfrak{c}$, there exists a countable dense subgroup $H$ of $G$ such that $S\cap H=\{e\}$. Clearly, each $d$-independent semitopological group is separable.

For a Tychonoff space $X$, denote the free abelian topological group $A(X)$ and free topological group $F(X)$ by $G(X)$. For each subgroup $H$ of $G(X)$, let $X_{H}$ be the minimal subset of $X$ such that $H\subset G(X_{H}, X)$.

\smallskip
Let $X$ be a space and $\mathcal{P}$ a family of subsets of $X$.  Then we say that

\smallskip
$\bullet$ $X$ is separable if there exists a countable subset $D$ of $X$ such that the closure of $D$ is $X$;

\smallskip
$\bullet$ $X$ has {\it countable tightness} if, for any subset $A$ of $X$ and any point $x\in X$ with $x\in \overline{A}$, there exists a countable subset $D$ of $A$ such that $x\in \overline{D}$;

\smallskip
$\bullet$ $\mathcal{P}$ is a {\it network} of $X$ if for each point $x\in X$ and any open subset $U$ with $x\in U$ there exits $P\in\mathcal{P}$ such that $x\in P\subset U$;

\smallskip
$\bullet$ $\mathcal{P}$ is a {\it $\pi$-network} of $X$ if for any nonempty open subset $U$ there exits $P\in\mathcal{P}$ such that $P\subset U$;

\smallskip
$\bullet$ $X$ is said to be {\it cosmic} if $X$ is regular and has countable network.

\smallskip
Recall that a subset $X$ of an abstract group $G$ is {\it unconditionally closed} in $G$ if $X$ is closed in every Hausdorff topological group topology on $G$. We say that $G$ is an {\it $M$-group} if each proper unconditionally closed subgroups of $G$ has index at least $\mathfrak{c}$. For an abelian $G$, it follows from \cite[Proposition 1.4]{DS2016} that it is an $M$-group iff, for every integer $m\geq 1$, either $|mG|=1$ or $|mG|\geq\mathfrak{c}$.

For a space $X$, {\it the compact covering number, cov$_{k}(X)$}, is the least cardinal $\tau$ such that $X$ can be covered by $\tau$ many compact sets; in \cite{DJ2020}, set $$d_{k}(X)=\min\{cov_{k}(Y): Y\ \mbox{is a dense subset of}\ X\}$$ and then $$\delta_{k}(X)=\sup\{d_{k}(Y): Y\ \mbox{is a dense subset of}\ X\}.$$ Recall that a space $X$ is {\it $k$-separable} if $d_{k}(X)\leq\aleph_{0}$ and {\it densely $k$-separable} \cite{DJ2020} if $\delta_{k}(X)\leq\aleph_{0}$.

\smallskip
If $X$ is a space, let $$\delta(X)=\sup\{d(Y): Y\ \mbox{is dense in}\ X\};$$ $\delta(X)$ is the {\it strong density of $X$} \cite{WS1976}. If $\delta(X)=\omega$, we say that $X$ is {\it dense-separable} \cite{LM1975}, that is, each dense subspace of $X$ is separable. Clearly, each hereditarily separable space is dense-separable, but not vice versa, see Example~\ref{r0} in Section 2. Moreover, from \cite{DJ2020}, it follows that there exists a $\sigma$-compact non-separable space such that $d_{k}(X)<d(X)$. However, the following question was posed in \cite{DJ2020}.

\begin{question}\cite[Question 1]{DJ2020}
Is there a regular Hausdorff space $X$ such that $\delta_{k}(X)<\delta(X)$?
\end{question}

 \maketitle
\section{Basic properties of dense-separable semitopological groups}
In this section, we mainly establish some basic properties of dense-separable semitopological groups. We start with two simple propositions, which show the class of dense-separable spaces is wide.

\begin{proposition}
Each monotonically normal separable space is dense-separable.
\end{proposition}

\begin{proof}
Since monotonically normal separable space is hereditarily separable \cite{G1997}, the result is obvious.
\end{proof}

\begin{proposition}\label{p5}
Each separable space $X$ with a countable tightness is dense-separable.
\end{proposition}

\begin{proof}
Let $H$ be an arbitrary dense subspace of $X$ and $D$ a countable dense subset of $X$. Since $X$ is of countable tightness, it follows that, for any $d\in D$, there exists a countable subset $H_{d}$ of $H$ such that $d\in \overline{H_{d}}$. Put $H^{\prime}=\bigcup_{d\in D}H_{d}\subset H$. Clearly, $H^{\prime}$ is countable. Then $X=\overline{D}\subset\bigcup_{d\in D}\overline{H_{d}}\subset\overline{\bigcup_{d\in D}H_{d}}=\overline{H^{\prime}}$, which shows that $H^{\prime}$ is a countable dense subset of $H$.
\end{proof}

Now we give some examples of dense-separable and non-dense-separable spaces respectively.

\begin{example}\label{r0}
(1) There exists a separable space with a countable tightness is not hereditarily separable.
Indeed, let $G$ be the Sorgenfrey line. Since $G\times G$ is a regular separable first-countable space, it follows from Proposition~\ref{p5} that $G\times G$ is dense-separable; however, it is well known that $G\times G$ is not hereditarily separable.

(2) Let $I$ be the closed unit interval, and let $X=I^{I}$ endowed with the product topology; then it follows from \cite[Example 5.3]{LM1975} that $X$ is a separable space but not dense-separable.

(3) For any separable metrizable space or separable and ordered space $X$, each compactification of $X$ is dense-separable, see \cite{LM1975}.

(4) In \cite[Theorem 3]{M2008}, E. Momtahan gave an algebraic characterization of a non-empty compact Hausdorff space $X$ such that it is dense-separable. Indeed, the author prove that, for any non-empty compact Hausdorff space $X$, it is dense-separable iff $C(X)$ has the following property: if $\{M_{i}\}_{i\in I}$ is a family of maximal ideals with $\bigcap_{i\in I}M_{i}=(0)$, then there exists a countable subset $I_{0}$ of $I$ such that $\bigcap_{i\in I_{0}}M_{i}=(0)$, where $C(X)$ is the ring of real-valued, continuous functions on a completely regular space $X$.

(5) It is well known that each hereditarily separable Moore space is metrizable; however, there exists a dense-separable Moore paratopological group is not metrizable, such as, the space in \cite[Example 4.1]{LL2012}.
\end{example}

It is well known that the product of no more than $\mathfrak{c}$ separable spaces is separable, see \cite[Theorem 2.3.15]{E1989}. However, it follows from (2) of Example~\ref{r0} that $I^{I}$ is separable and not dense-separable, where $I$ is the closed unit interval. Indeed, we also have the following example.

\begin{example}
There exists a separable non-compact topological group which is not dense-separable.
\end{example}

\begin{proof}
Let $A(I^{I})$ be the free abelian topological group. Then it follows from \cite[Lemma 3.1]{CI1977} that $A(I^{I})$ is separable and non-compact. However, since $I^{I}$ is not dense-separable, there exists a dense subset $X$ of $I^{I}$ such that $X$ is not separable. Then the free abelian topological group $A(X)$ is a topological subgroup of $A(I^{I})$ by \cite[Lemma 7.7.2 and Theorem 7.7.4]{AT2008}. Clearly, $A(X)$ is dense in $A(I^{I})$. However, it follows from \cite[Lemma 3.1]{CI1977} again that $A(X)$ is not separable. Therefore, $A(I^{I})$ is not dense-separable.
\end{proof}

\begin{example}\label{2022eee}
There exists a dense-separable topological group $G$ which is not hereditarily separable.
\end{example}

\begin{proof}
Let $X$ be the Michael line, see \cite[Example 1.8.5]{LY2016}. Then $X^{n}$ is Lindel\"{o}f for each $n\in\mathbb{N}$ by \cite[Theorem 1]{L1992}. Moreover, it is obvious that $X$ is submetrizable. Therefore, $G=C_{p}(X)$ is separable and has a countable tightness by
\cite[Corollaries 5.1.7 and ~5.4.3]{Shoulin2004}, thus it is dense-separable by Proposition~\ref{p5}. However, $X$ is not hereditarily Lindel\"{o}f, hence it follows from \cite[Theorem 6]{Z1980} that $G$ is not hereditarily separable.
\end{proof}

The following proposition and corollary are obvious.

\begin{proposition}
Each dense subset of a dense-separable space is also dense-separable.
\end{proposition}

\begin{corollary}
If $X$ and $Y$ are two spaces such that $X\times Y$ is dense-separable, then $X$ and $Y$ are dense-separable.
\end{corollary}

Let $X, Y$ be two spaces and $f: X\rightarrow Y$ a mapping. We say that $f$ is {\it almost open} if the interior of $f(U)$ is nonempty in $Y$ for each nonempty open subset $U$ of $X$.

\begin{proposition}\label{p7}
Let $f: X\rightarrow Y$ be an almost open continuous mapping. If $X$ is dense-separable, then $Y$ is dense-separable.
\end{proposition}

\begin{proof}
Assume that $X$ is dense-separable. Take any dense subset $D$ of $Y$. Since $f$ is almost open, it is easy to see that $f^{-1}(D)$ is dense in $X$, then there exists a countable subset $F$ of $X$ such that $F$ is dense in $f^{-1}(D)$. Then $f(F)$ is a countable dense subset of $D$. Hence $Y$ is dense-separable.
\end{proof}

\begin{corollary}\label{c1}
Let $G$ be a dense-separable topological group and $H$ is a closed subgroup of $G$. Then the naturally quotient space $G/H$ is dense-separable.
\end{corollary}

\begin{proof}
From \cite[Theorem 1.5.1]{AT2008}, it follows that the naturally quotient mapping is open, hence $G/H$ is dense-separable by Proposition~\ref{p7}.
\end{proof}

\begin{proposition}\label{pp2}
Let $X$ be a dense-separable space. Then each open subspace is dense-separable.
\end{proposition}

\begin{proof}
Let $U$ be an arbitrary open set of $X$, and take any dense subset $Y$ of $U$. Then $Y\cup (X\setminus U)$ is a dense of $X$. Since $X$ is dense-separable, there exists a countable dense-subset $D$ of $Y\cup (X\setminus U)$. Put $D_{1}=D\cap Y$. Clearly, it is easy to see that $D_{1}$ is nonempty and dense in $Y$, hence $D_{1}$ is dense in $U$.
\end{proof}

\begin{proposition}
Let $X$ be a dense-separable space. If each closed subspace of $X$ is dense-separable, then $X$ is hereditarily separable.
\end{proposition}

\begin{proof}
Let $Y$ be an arbitrary nonempty subspace of $X$. Then $\overline{Y}$ is dense separable, hence there exists a countable subset $D$ of $Y$ such that $D$ is dense $Y$. Therefore, $X$ is hereditarily separable.
\end{proof}

By (1) of Example~\ref{r0}, there exists a closed subspace of a dense-separable space is not separable. However, the space $G\times G$ in (1) of Example~\ref{r0} is not normal, hence we have the following question.

\begin{question}\label{2022qqq}
If $G$ is a normal (or perfect normal) dense-separable space, is $G$ hereditarily separable? What if $G$ is a topological group?
\end{question}

\begin{remark}
If topological group $G$ in Example~\ref{2022eee} is normal, then Question~\ref{2022qqq} is negative.
\end{remark}

Moreover, it is natural to ask that if $\{G_{\alpha}: \alpha<\mathfrak{c}\}$ is a family of dense-separable topological groups, is the product $\prod_{\alpha<\mathfrak{c}}G_{\alpha}$ dense-separable? Indeed, we have the following proposition.

\begin{proposition}\label{p6}
For any a family $\{G_{\alpha}: \alpha<\omega_{1}\}$ of non-trivial Hausdorff spaces, the dense $\Sigma$-product $\Sigma\Pi_{\alpha\in\omega_{1}}G_{\alpha}$ with basic point $b=(b_{\alpha})$ is not separable, where the $\Sigma$-product of $\{G_{\alpha}: \alpha<\omega_{1}\}$ with basic point $b$ is the subspace of $\Pi_{\alpha\in\omega_{1}}G_{\alpha}$ consisting of all points $x\in X$ such that only countably many coordinates $g_{\alpha}$ of $g$ are distinct from the corresponding coordinates $b_{\alpha}$ of $b$.
\end{proposition}

\begin{proof}
Assume that $\Sigma\Pi_{\alpha\in\omega_{1}}G_{\alpha}$ is separable, then there exists a countable dense subset $D$ of $\Sigma\Pi_{\alpha\in\omega_{1}}G_{\alpha}$. For each $d=(d_{\alpha})\in D$, there exists a countable subset $A_{d}\subset\omega_{1}$ such that $d_{\alpha}=b_{\alpha}$ for each $\alpha\in\omega_{1}\setminus A_{d}$. Put $A=\bigcup_{d\in D}A_{d}$; then $A$ is countable. Take any $\beta\in \omega_{1}\setminus A$ and take $h=(h_{\alpha})\in\Sigma\Pi_{\alpha\in\omega_{1}}G_{\alpha}$ with $h_{\beta}\neq b_{\beta}$ and $h_{\alpha}=b_{\alpha}$ for any $\alpha\neq\beta$. By the Hausdorff property, choose any open neighborhood $V$ of $h_{\beta}$ in $G_{\beta}$ such that $b_{\beta}\not\in V$. Then $U=V\times \Pi_{\alpha\in\omega_{1}\setminus\{\beta\}}G_{\alpha}$ is an open neighborhood of $h$ with $U\cap D=\emptyset$, which is a contradiction.
\end{proof}

A semitopological group $G$ is said to be {\it dense-subgroup-separable} if each dense subgroup is separable. Clearly, each dense-separable semitopological group is dense-subgroup-separable. However, the following question is still unknown for us.

\begin{question}\label{q2}
Let $G$ be a dense-subgroup-separable semitopological group. Is $G$ dense-separable? What if $G$ is a topological group?
\end{question}

\begin{corollary}
The compact group $D^{\omega_{1}}$ is not dense-subgroup-separable, where $D=\{0, 1\}$ is the Boolean group with the discrete topology.
\end{corollary}

\begin{corollary}
Any Hausdorff space $X$ which contains a dense subspace that is homeomorphic to $\Sigma D^{\omega_{1}}$ is not dense-separable, where $D=\{0, 1\}$ is the Boolean group with the discrete topology.
\end{corollary}

\begin{corollary}
For any a family $\{G_{\alpha}: \alpha<\kappa\}$ of non-trivial separable metrizable spaces, the product space $\Sigma_{\alpha\in\kappa}G_{\alpha}$ is dense-separable iff $\kappa<\omega_{1}$ iff $\Pi_{\alpha\in\kappa}G_{\alpha}$ is metrizable.
\end{corollary}

In \cite{WS1976}, the authors gave a dense-separable Hausdorff (non-regular) space $X$ such that $X\times X$ is not dense-separable. However, the following two questions are interesting.

\begin{question}\label{q1}
Let $X$ and $Y$ be two dense-separable regular spaces. Is the product $X\times Y$ dense-separable? What if $X$ and $Y$ are hereditarily separable?
\end{question}

\begin{question}\label{q7}
Let $X$ and $Y$ be two dense-separable Hausdorff topological groups. Is the product $X\times Y$ dense-separable? What if $X$ and $Y$ are hereditarily separable?
\end{question}

A space $X$ is said to be {\it right} (resp., {\it left}) {\it separated} if there exist a well-ordering $\prec$ of $X$ and neighborhoods $U_{x}$ for each $x\in X$ such that $x\prec y$ (resp., $y\prec x$) implies $y\not\in U_{x}$. Next we give some partial answers to Question~\ref{q1}.

\begin{proposition}\label{p4}
Let $X$ be a separable regular space. Then $X$ is dense-separable iff it has no uncountable left separated dense subspace.
\end{proposition}

\begin{proof}
Since $X$ is a separable regular space, it follows from \cite[Theorem 3.3]{H1984} that $w(X)\leq\mathfrak{c}$. An uncountable left separated dense subspace is not separable, so dense-separable space has no uncountable left separated dense subspace. Suppose that $X$ is not dense-separable, and assume that $Y$ is a non-separable dense subspace of $X$. Since $w(X)\leq\mathfrak{c}$, we can define a maximal subset of points $\{x_{\alpha}: \alpha<\mathfrak{c}\}$ in $Y$ such that $x_{\alpha}\not\in\overline{\{x_{\beta}: \beta<\alpha\}}$ for each $\alpha<\mathfrak{c}$. Since $\{x_{\alpha}: \alpha<\mathfrak{c}\}$ is maximal, it is easy to see that $\{x_{\alpha}: \alpha<\mathfrak{c}\}$ is dense in $Y$, which is a contradiction.
\end{proof}

\begin{theorem}\label{t11}
Let $X$ be a dense-separable regular space and $Y$ be a cosmic space. Then $X\times Y$ is dense-separable.
\end{theorem}

\begin{proof}
Let $\{W_{n}: n\in\mathbb{N}\}$ be a countable network of $Y$. Assume that $X\times Y$ is not dense-separable. By Proposition~\ref{p4}, suppose that $\{(x_{\alpha}, y_{\alpha}): \alpha<\omega_{1}\}$ is a left separated dense subspace of $X\times Y$, so there exist separating neighborhoods $\{U_{\alpha}: \alpha<\omega_{1}\}$ such that $(x_{\beta}, y_{\beta})\in U_{\beta}$ for each $\beta\in\omega_{1}$, but $(x_{\alpha}, y_{\alpha})\not\in U_{\beta}$ for any $\alpha<\beta$. Without loss of generality, we may assume that each $U_{\alpha}$ has the form $V_{\alpha}\times O_{\alpha}$, where each $V_{\alpha}$ and $O_{\alpha}$ are open in $X$ and $Y$ respectively. For each $\alpha\in\omega_{1}$, there exists a subset $N_{\alpha}\subset\mathbb{N}$ such that $O_{\alpha}=\bigcup_{i\in N_{\alpha}}W_{i}$. However, $\bigcup_{\alpha\in\omega_{1}}N_{\alpha}$ is countable, hence there exist an uncountable set $I\subset\omega_{1}$ and $m\in \bigcup_{\alpha\in\omega_{1}}N_{\alpha}$ such that $(x_{\alpha}, y_{\alpha})\in V_{\alpha}\times W_{m}$ and $m\in N_{\alpha}$ for each $\alpha\in I$. Clearly, $y_{\alpha}\in W_{m}$ for each $\alpha\in I$. Take any $\alpha, \beta\in I$ with $\alpha<\beta$. Then $x_{\beta}\in V_{\beta}$ and $x_{\alpha}\not\in V_{\beta}$; otherwise, we have $(x_{\alpha}, y_{\alpha})\in V_{\beta}\times W_{m}\subset V_{\beta}\times O_{\beta}=U_{\beta}$, which shows that $\{x_{\alpha}: \alpha\in I\}$ is an uncountable left separated subspace. If $\{x_{\alpha}: \alpha\in I\}$ is dense in $X$, then it follows from Proposition~
\ref{p4} that $X$ is not dense-separable, which is a contradiction. Assume that $\{x_{\alpha}: \alpha\in I\}$ is not dense in $X$. Since $w(X)\leq \mathfrak{c}$, there exists a maximal subset $B=\{z_{\delta}: \delta<\gamma\}$ of $X$ such that $\gamma\leq\mathfrak{c}$ and $z_{\alpha}\not\in \overline{A\cup\{z_{\delta}: \delta<\alpha\}}$ for any $\alpha<\gamma$. Then $A\cup B$ is left separated subspace of $X$ and dense in $X$. From Proposition~\ref{p4} again, it follows that $X$ is not dense-separable, which is a contradiction.
\end{proof}

\begin{corollary}
Let $\{X_{n}: n\in\mathbb{N}\}$ be a sequence of cosmic spaces. Then $\Pi_{i\in B}X_{i}$ is dense-separable.
\end{corollary}

\begin{theorem}\label{t44}
Suppose that $\{X_{n}: n\in\mathbb{N}\}$ is a sequence of spaces  such that $\Pi_{i\in B}X_{i}$ is dense-separable for any finite subset $B$ of $\mathbb{N}$. Then $\Pi_{i\in\mathbb{N}}X_{i}$ is dense-separable.
\end{theorem}

\begin{proof}
Take any dense subset $Y$ of $\Pi_{i\in\mathbb{N}}X_{i}$. Let $\mathcal{A}$ denote the collection of finite subsets of $\mathbb{N}$. For each $B\in\mathcal{A}$, it is obvious that $\pi_{B}(Y)$ is dense in $\Pi_{i\in B}X_{i}$, where $\pi_{B}$ is the natural projection from $\Pi_{i\in B}X_{i}$ to $\Pi_{i\in B}X_{i}$; hence take a countable subset $Y_{B}$ of $Y$ such that $\pi_{B}(Y_{B})$ is dense in $\pi_{B}(Y)$. Put $A=\bigcup_{A\in\mathcal{A}}Y_{B}$. Then $A$ is a countable dense subset of $Y$.
\end{proof}

From Proposition~\ref{p5}, it is obvious that we have the following proposition.

\begin{proposition}
Let $G$ be a semitopological group with a countable tightness. Then $G$ is dense-separable iff there exists a dense subgroup of $G$ is separable iff each dense subgroup of $G$ is separable.
\end{proposition}

Let $G$ be a topological group. Then it is well known that if $G$ is separable then the Ra\u{\i}kov completion of $G$ is separable, but not vice versa. However, if the Ra\u{\i}kov completion of $G$ is dense-separable, then $G$ is dense-separable, but not vice versa. Indeed, the topological group $\mathbb{T}^{\mathfrak{c}}$ is separable and not dense-separable by Proposition~\ref{p6}, where $\mathbb{T}$ is the unit circle with usual topology. Take any countable dense subset $X$ of $\mathbb{T}^{\mathfrak{c}}$, and let $H=\langle X\rangle$ be the subgroup generated by $X$. Then $H$ is countable and dense in $\mathbb{T}^{\mathfrak{c}}$, thus it is dense-separable. Therefore, we have the following question.

\begin{question}
Let $G$ be a dense-separable topological group. When is the Ra\u{\i}kov completion of $G$ dense-separable?
\end{question}

It is well-known that each compact hereditarily separable topological group is metrizable. Indeed, we can generalize ``hereditarily separable'' to ``dense-separable'' by the following Theorem~\ref{tt1}. First, we introduce the following concept.

We say that a space $X$ is {\it locally dense-separable} (resp., {\it locally hereditarily separable}) if, for each point $x\in X$, there exists an open neighborhood $U$ of $x$ in $X$ such that $U$ is dense-separable (resp., hereditarily separable).

\begin{theorem}\label{tt1}
For a locally compact topological group $G$, the following statements are equivalent:
\begin{enumerate}
\item $G$ is locally dense-separable;

\item $G$ is locally hereditarily separable;

\item $G$ is metrizable.
\end{enumerate}
\end{theorem}

\begin{proof}
It suffices to prove (1) $\Rightarrow$ (3). Indeed, since $G$ it locally compact and locally dense-separable, there exist open neighborhoods $U$ and $W$ of $e$ in $G$  such that $\overline{U}$ is compact and $W$ is dense-separable. Then $O=W\cap U$ is dense-separable by Proposition~\ref{pp2}. Clearly, $\overline{O}$ is compact and dense-separable, then it follows from \cite{JS1989} that $\pi(\overline{O})=\omega$, thus $G$ is metrizable by
\cite[Theorem 3.3.12 and Proposition~5.2.6]{AT2008}.
\end{proof}

From the proof of Theorem~\ref{tt1}, we also have the following proposition.

\begin{proposition}
For a locally compact space $X$, if $X$ is locally dense-separable, then $\pi_{\chi}(X)=\omega$.
\end{proposition}

\begin{corollary}
For a locally compact topological group $G$, the following statements are equivalent:
\begin{enumerate}
\item $G$ is dense-separable;

\item $G$ is hereditarily separable;

\item $G$ is separable metrizable.
\end{enumerate}
\end{corollary}

\begin{corollary}
The compact group $\mathbb{T}^{\mathfrak{c}}$ is not dense-separable.
\end{corollary}

Since $\mathbb{T}^{\mathfrak{c}}$ is separable, there exists a countable subset $C$ of $\mathbb{T}^{\mathfrak{c}}$ such that $C$ is dense in $\mathbb{T}^{\mathfrak{c}}$, hence the subgroup $H$ generated by $C$ is a countable dense subgroup of $\mathbb{T}^{\mathfrak{c}}$. Clearly, $H$ is not metrizable; however, $H$ is hereditarily separable, thus dense-separable. Therefore, we have the following interesting two questions.

\begin{question}
If $G$ is a dense-separable subset (or subgroup) of $\mathbb{T}^{\mathfrak{c}}$, is $G$ hereditarily separable?
\end{question}

\begin{question}
How to characterize the subgroup $G$ of $\mathbb{T}^{\mathfrak{c}}$ such that $G$ is dense-separable?
\end{question}

A topological group $G$ is {\it feather} if it contains a non-empty compact set $K$ of countable character in $G$. Since each locally compact group is feather, we can generalize Theorem~\ref{tt1} as follows.

\begin{theorem}
If $G$ is feather dense-separable topological group, then $G$ is metrizable.
\end{theorem}

\begin{proof}
Since $G$ is a feather topological group, it follows from \cite[Section 4.3]{AT2008} that there exists a compact subgroup $K$ such that the quotient space $G/K$ is metrizable. Moreover, $G$ is a cosmic space by \cite[Theorem 4]{T2010}, hence $K$ is metrizable. Therefore, $G$ is metrizable by \cite[Corollary 1.5.21]{AT2008}.
\end{proof}

Recall that a space $X$ is {\it initially $\omega_{1}$-compact} if every open cover of size $\leq \omega_{1}$ has
a finite subcover. In \cite{A1994}, the author proved that each initially $\omega_{1}$-compact hereditarily separable topological group is metrizable. If $X$ is a separable regular space, then $w(X)\leq\mathfrak{c}$, hence each separable regular initially $\omega_{1}$-compact space is compact. Therefore, we have corollary by Theorem~\ref{tt1}.

\begin{corollary}\label{cc}
Let $G$ be a locally initially $\omega_{1}$-compact topological group. If $G$ is locally dense-separable, then $G$ is metrizable.
\end{corollary}

However, the authors in \cite{JKS2009} force a first-countable, locally compact, initially $\omega_{1}$-compact space $X$ of size $\omega_{1}$ which is not compact; hence $\pi(X)<\delta(X)$ by Corollary~\ref{cc}.

\begin{question}\label{q4}
Under CH, if $X$ is an initially $\omega_{1}$-compact space, then does $\pi(X)=\delta(X)$ hold?
\end{question}

A semitopological group $G$ is {\it precompact} (resp., {\it $\omega$-narrow}) if for each open set $U$ of $G$ there exists a finite set (resp., countable set) $A\subset G$ such that $AU=UA=G$. Since each locally precompact group $G$ can embed as a dense topological subgroup to a locally compact group \cite{W1937}, it follows from Theorem~\ref{tt1} that we have the following corollary.

\begin{corollary}
Let $G$ be a locally precompact topological group. If the Ra\u{\i}kov completion of $G$ is locally dense-separable, then $G$ is metrizable.
\end{corollary}

It is well known that under the assumption of set-theory there exists a pseudocompact, hereditarily separable, non-metrizable topological group $G$ (see \cite{DD2005} or \cite{HJ1976}), thus $G$ is dense-separable; then the the Ra\u{\i}kov completion of $G$ is not dense-separable by Theorem~\ref{tt1}. However, the following question is interesting.

\begin{question}\label{q5}
Which an abelian group $G$ admits a (countably compact or pseudocompact) dense-separable group topology?
\end{question}

By \cite[Theorems~2.1, 2.6 and 2.7]{DD2005}, the following theorem holds, which gives a partial answer to Question~\ref{q5}.

\begin{theorem}
Under $\nabla_{\kappa}$ (see \cite{DD2005}), the following conditions are equivalent for any abelian group $G$:
\begin{enumerate}
\item $G$ admits a separable ({\it resp., separable pseudocompact, separable countably compact}) group topology;

\item $G$ admits a dense-separable ({\it resp., dense-separable pseudocompact, dense-separable countably compact})  group topology;

\item $G$ admits a hereditarily separable ({\it resp., hereditarily separable pseudocompact, hereditarily separable countably compact}) group topology.
\end{enumerate}

\end{theorem}

\begin{question}\label{q6}
If $G$ is a countably compact or pseudocompact, dense-separable topological group, is $G$ hereditarily separable?
\end{question}

Obviously, for a dense-separable topological group $G$, each closed subgroup of $G$ is dense-separable iff each subgroup of $G$ is dense-separable. By (1) of Example~\ref{r0}, there exists a closed subgroup $H$ of a dense-separable paratopological group $G$ such that $H$ is not separable. We conjecture the following question is negative.

\begin{question}
Let $G$ be a dense-separable topological group. Is each closed subgroup $H$ of $G$ dense-separable?
\end{question}

In \cite[Theorem 3.2]{LMT2017}, the authors proved that for an $\omega$-narrow topological group $G$, it is homeomorphic to a subspace of a separable regular space iff $G$ is topologically isomorphic to a subgroup of a separable topological group. Therefore, we have the following question.

\begin{question}
Let $G$ be an $\omega$-narrow topological group. If $G$ is homeomorphic to a subspace of a dense-separable regular space, is $G$ topologically isomorphic to a subgroup of a dense-separable topological group?
\end{question}

Let $X, Y$ be two spaces and $f: X\rightarrow Y$ a mapping. We say that $f$ is {\it perfect} if $f^{-1}(y)$ is compact for each $y\in Y$ and $f(F)$ is closed in $Y$ for each closed subset $F$ of $X$.

\begin{proposition}
Let $G$ be a compact topological group and $K$ a closed dense-separable normal subgroup. Then $G$ is dense-separable iff $G/H$ is dense-separable.
\end{proposition}

\begin{proof}
Let $f: G\rightarrow G/H$ be the naturally quotient mapping. By Corollary~\ref{c1}, the necessity is obvious. Suppose that $G/H$ is dense-separable. Clearly, $G/H$ is compact. Hence it follows from Theorem~\ref{tt1} that $K$ and $G/K$ are separable metrizable. Since $f$ is a perfect mapping, it follows from \cite[Corollary 3.3.20]{AT2008} that $G$ is metrizable, thus $G$ is dense-separable.
\end{proof}

However, the following question is still unknown for us.

\begin{question}
Let $G$ be a topological group and $K$ a compact dense-separable normal subgroup. If $G/H$ is dense-separable, is $G$ dense-separable?
\end{question}

 \maketitle
\section{dense-subgroup-separable topological groups}
In this section, we mainly study dense-subgroup-separable topological groups, and give some partial answers to the following Questions~\ref{q3} and ~\ref{q8} respectively.

\begin{question}\label{q3}
How to characterize a dense-subgroup-separable topological group $G$? What if $G$ is compact?
\end{question}

The following Theorem~\ref{t22} gives a characterization for a Tychonoff space $X$ such that $X$ is dense-separable, which gives a partial answer to Questions~\ref{q2} and~\ref{q3}.

\begin{theorem}\label{t22}
Let $X$ be a Tychonoff space. Then $X$ is dense-separable iff, for each dense subgroup $H$ of $G(X)$, there exists a countable subset of $X_{H}$ such that $G(X_{H}, X)$ is dense in $G(X)$.
\end{theorem}

\begin{proof}
Necessity. Assume that $X$ is dense-separable. Take any dense subgroup $H$ of $G(X)$. We claim that $X_{H}$ is dense in $X$. Suppose not, then $\overline{X_{H}}$ is a proper closed subset of $X$. From \cite[Theorem 7.4.5]{AT2008}, it follows that $G(\overline{X_{H}}, X)$ is non-empty and closed in $G(X)$. Then $G(X)\setminus G(\overline{X_{H}}, X)$ is a non-empty open set in $G(X)$, but $(G(X)\setminus G(\overline{X_{H}}, X))\cap H=\emptyset$, which is a contradiction. Therefore, $X_{H}$ is dense in $X$. Since $X$ is dense-separable, there exists a countable subset $D$ of $X_{H}$ such that $D$ is dense in $X_{H}$, thus it is dense in $X$. Then it is obvious that $G(D, X)$ is a dense in $G(X)$.

Sufficiency. Take any dense subset $D$ of $X$. Then $G(D, X)$ is dense in $G(X)$. Clearly, $X_{G(D_{0}, X)}=D_{0}$. From our assumption, there exists a countable subset of $D_{0}$ such that $G(D_{0}, X)$ is dense in $G(X)$. We claim that $D_{0}$ is dense in $D$. Suppose not, $G(\overline{D_{0}}, X)$ is a closed proper subgroup of $G(X)$. Obviously, $G(D_{0}, X)\subset G(\overline{D_{0}}, X)$, which is a contradiction with the density of $G(D_{0}, X)$ in $G(X)$. Therefore, $D_{0}$ is a countable dense subset of $D$.
\end{proof}

By Theorem~\ref{t22}, we have the following corollary.

\begin{corollary}
Assume that $X$ is a Tychonoff space. If $G(X)$ is dense-subgroup-separable, then $X$ is dense-separable.
\end{corollary}

We conjecture the following question is positive which gives a partial answer to Question~\ref{q3}.

\begin{question}\label{q8}
Let $G$ be a compact dense-subgroup-separable topological group. Is $G$ dense-separable?
\end{question}

Next we give some partial answers to Question~\ref{q8}.

\begin{theorem}\label{t33333}
Let $G$ be a compact abelian torsion topological group. If $G$ is dense-subgroup-separable, then $G$ is metrizable; thus $G$ is dense-separable.
\end{theorem}

\begin{proof}
Assume that $G$ is dense-subgroup-separable. Since $G$ is a compact abelian torsion topological group, it follows from \cite[Corollary 4.2.2]{DPS1990} that $G$ is topologically isomorphic to a product of finite cyclic groups. Let $G=\prod_{i\in I} G_{i}$, where each $G_{i}$ is a finite cyclic group with a discrete topology. Then $I$ is at most countable by Proposition~\ref{p6}, hence $G$ is metrizable.
\end{proof}

\begin{theorem}\label{t44444}
Let $G$ be a compact abelian torsion-free topological group. If $G$ is dense-subgroup-separable, then $G$ is metrizable; thus $G$ is dense-separable.
\end{theorem}

\begin{proof}
Assume that $G$ is dense-subgroup-separable. Since $G$ is a compact abelian torsion-free topological group, it follows from \cite[Corollary 8.5]{HM1998} that there exists a family of sets $\{X_{p}: p\in\{0\}\cup\mathcal{P}\}$ such that $G$ is topologically isomorphic to $(\widehat{\mathbb{Q}})^{X_{0}}\times \prod_{p\in\mathcal{P}}\mathbb{Z}_{p}^{X_{p}}$, where $\mathcal{P}$ denotes the set of all prime
numbers. Since $G$ is dense-subgroup-separable, it follows from Proposition~\ref{p6} that each $X_{p}$ is countable for each $p\in\{0\}\cup\mathcal{P}$. Therefore, $G$ is metrizable.
\end{proof}

A space $X$ is said to be {\it arcwise connected} if for every pair $x_{1}, x_{2}$ of distinct points of $X$ there exists a homeomorphic embedding $h: [0,1]\rightarrow X$ such that $h(0)=x_{1}$ and $h(1)=x_{2}$.

\begin{theorem}
Let $G$ be an arcwise connected abelian group. If $G$ is dense-subgroup-separable, then $G$ is metrizable; thus $G$ is dense-separable.
\end{theorem}

\begin{proof}
By Torus Proposition (see \cite[page of 406]{HM1998}, $G$ is torus, that is, $G$ is topologically isomorphic to $\mathbb{T}^{\kappa}$. Since $G$ is dense-subgroup-separable, it follows from Proposition~\ref{p6} that $K=\aleph_{0}$, hence $G$ is metrizable.
\end{proof}

\begin{theorem}\label{t444444}
Let $G$ be a compact abelian connected topological group. If $G$ is dense-subgroup-separable, then $G$ is metrizable; thus $G$ is dense-separable.
\end{theorem}

\begin{proof}
Assume that $G$ is dense-subgroup-separable. Since $G$ is a compact abelian connected topological group, it follows from \cite[Corollary 8.18]{HM1998} that there is arbitrary small 0-dimension compact subgroup $D$ such that $G/D$ is a torus. Obviously, the natural mapping $\pi: G\rightarrow G/D$ is open, perfect and continuous, hence it is easy to see that $G/D$ is dense-subgroup-separable. Since $G/D$ is a torus, there exists a set $X$ such that $G/D$ is topologically isomorphic to $\mathbb{T}^{X}$. Then $X$ is countable by Proposition~\ref{p6} since $G/D$ is dense-subgroup-separable. Then $G/D$ is metrizable, hence $G$ is also metrizable since $f$ is a perfect mapping.
\end{proof}

From \cite[Corollary 8.5]{HM1998}, a compact abelian group is connected iff it is divisible, hence we have the following corollary.

\begin{corollary}
Let $G$ be a compact abelian divisible topological group. If $G$ is dense-subgroup-separable, then $G$ is metrizable; thus $G$ is dense-separable.
\end{corollary}

By Theorems~\ref{t33333}, ~\ref{t44444} and~\ref{t444444} and \cite[Theorem 8.45]{HM1998}, we have the following corollary.

\begin{corollary}
Let $G$ be a compact torsion (or divisible, or torsion-free) abelian group. Then the following conditions are equivalent:
\begin{enumerate}
\smallskip
\item $G$ is dense-subgroup-separable;

\smallskip
\item $G$ is dense-separable;

\smallskip
\item $\widehat{G}$ is countable, where $\widehat{G}$ is the dual group of $G$ which is endowed with the compact-open topology;

\smallskip
\item $G$ is topologically isomorphic to a subgroup of the torus $\mathbb{T}^{\aleph_{0}}$.
\end{enumerate}
\end{corollary}

\begin{theorem}\label{t5555555}
Let $G$ be a compact abelian totally disconnected topological group. If $G$ is dense-subgroup-separable, then $G$ is metrizable; thus $G$ is dense-separable.
\end{theorem}

\begin{proof}
From \cite[Corollary 8.8]{HM1998}, it follows that $G$ is the product of $\prod_{p\in\mathcal{P}}G_{p}$ of its $p$-primary components. Since $G$ is dense-subgroup-separable, it follows that each $G_{p}$ is dense-subgroup-separable, then each $G_{p}$ is metrizable by Theorem~\ref{t33333}. Therefore, $G$ is metrizable.
\end{proof}

An abstract group is called {\it simple} if it has no proper non-trivial normal
subgroup. A topological group G is said to be {\it topologically simple} if it has no proper nontrivial
closed normal subgroup. Since the connected component of the identity of
any topological group is a closed normal subgroup, it follows that each topologically simple
group is either totally disconnected or connected. Therefore, we have the following corollary by Theorems~\ref{t444444} and ~\ref{t5555555}.

\begin{corollary}
Let $G$ be a compact abelian topological simple group. If $G$ is dense-subgroup-separable, then $G$ is metrizable; thus $G$ is dense-separable.
\end{corollary}

\begin{corollary}
The compact group $\mathbb{T}^{\mathfrak{c}}$ is not dense-subgroup-separable.
\end{corollary}

However, the following two questions are still unknown for us.

\begin{question}
If $G$ is a dense-subgroup-separable subgroup of $\mathbb{T}^{\mathfrak{c}}$, is $G$ hereditarily separable?
\end{question}

\begin{question}
How to characterize the subgroup $G$ of $\mathbb{T}^{\mathfrak{c}}$ such that $G$ is dense-subgroup-separable?
\end{question}

\begin{proposition}
Assume $G$ is a subgroup of $\mathbb{T}^{\mathfrak{c}}$ such that $G$ is dense-subgroup-separable and each uncountable subgroup of $G$ is dense in $G$. If $A$ is a subgroup of $\mathbb{T}^{\mathfrak{c}}$ such that $|A|<|G|$, then $H=AG$ is dense-subgroup-separable.
\end{proposition}

\begin{proof}
Let $C$ be any dense subgroup of $H$. If $C$ is countable, we are done. Now let $C$ be uncountable. Clearly, we can take $a\in A$ and uncountable subset $B$ of $G$ such that $aB\subset C$. Then $B^{-1}B\subset G\cap C$, hence $G\cap C$ is uncountable. By our assumption, $G\cap C$ is a dense subgroup in $G$. Since $G$ is dense-subgroup-separable subgroup, $G\cap C$ is separable, hence $C$ is separable. Therefore, $H=AG$ is dense-subgroup-separable.
\end{proof}

By a similar proof of Theorem~\ref{t44}, we have the following theorem.

\begin{theorem}\label{tt100}
Suppose that $\{G_{n}: n\in\mathbb{N}\}$ is a sequence of semitopological groups  such that if $B$ is a finite subset of $\mathbb{N}$, then $\prod_{i\in B}G_{i}$ is dense-subgroup-separable. Then $\prod_{i\in\mathbb{N}}G_{i}$ is dense-subgroup-separable.
\end{theorem}

The following question is still unknown for us.

\begin{question}
Let $G$ and $H$ be two dense-subgroup-separable semitopological groups. Is $G\times H$ dense-subgroup-separable? What if $G$ and $H$ are topological groups?
\end{question}

 \maketitle
\section{Some applications of dense-separability in the $d$-independent of semitopological groups}
In this section, we discuss some applications of the dense-separability in $d$-independent topological groups and related structures. We start with the following obvious proposition.

\begin{proposition}\label{2022p4}
Every $d$-independent semitopological abelian group is algebraically an $M$-group.
\end{proposition}

An {\it almost topological group} \cite{e41} is a paratopological group $(G,\tau)$ which satisfies the following conditions:

\smallskip
(a) the group $G$ admits a Hausdorff topological group topology $\gamma$ weaker than $\tau$, and

\smallskip
(b) there exists a local base $\mathcal{B}$ at the neutral element $e$ of paratopological group $(G,\tau)$ such that the set $V=U\setminus{e}$ is open in $(G,\gamma)$ for each $U\in\mathcal{B}$. \\
Moreover, we will say that $G$ is an {\it almost topological group with structure $(\tau, \gamma, \mathscr{B})$}.

\begin{proposition}\label{2022p0}
Let $G$ be an almost topological group with structure $(\tau, \gamma, \mathscr{B})$. Then $(G, \gamma)$ is $d$-independent iff $(G, \tau)$ is $d$-independent.
\end{proposition}

\begin{proof}
The sufficiency is obvious. It suffices to consider the necessity. Suppose $(G, \gamma)$ is $d$-independent. Take any subgroup $H$ of $(G, \tau)$ with $|H|<\mathfrak{c}$. Since $(G, \gamma)$ is $d$-independent, there exists a countable dense subgroup $D$ of $(G, \gamma)$ such that $D\cap H=\{e\}$. By the proof of \cite[Theorem 4.4]{LLL2015} that $D$ is a dense in $(G, \tau)$. Therefore, $(G, \tau)$ is $d$-independent.
\end{proof}

\begin{proposition}
Let $G$ be a semitopological group. If $K$ is a dense subgroup of $G$ such that $K$ is $d$-independent, then $G$ is $d$-independent.
\end{proposition}

\begin{proof}
Let $S$ be an arbitrary subgroup of $G$ with $|S|<\mathfrak{c}$. Then $S_{K}=K\cap S$ is a subgroup of $K$. Clearly, $|S_{K}|<\mathfrak{c}$, hence there exists a countable dense subgroup $H$ of $K$ such that $H\cap S_{K}=\{e\}$. Obviously, $H$ is dense in $G$ and $H\cap S=\{e\}$. Therefore, $G$ is $d$-independent.
\end{proof}

\begin{remark}
The following Examples~\ref{2022e0} and~\ref{2022e1} show that there exist $d$-independent paratopological groups which are not topological groups. First, we give some lemmas.
\end{remark}

\begin{lemma}\label{2022l0}
Let $G$ be a finite discrete topological group and $\kappa$ be a cardinal satisfying $\omega\leq\kappa\leq\mathfrak{c}$. Then the compact group $G^{\kappa}$ has a countable $\pi$-network $\mathcal{W}$  such that  each $W\in\mathcal{W}$ contains a copy of $G^{\omega}$.
\end{lemma}

\begin{proof}
Let $G=\{g_{1}, \ldots, g_{N}\}$, where $N\in\mathbb{N}$. Identify $\kappa$ with a dense subset of the open interval $(0, 1)$. Assume that $\mathcal{F}$ is the countable family consisting of the sets of the form $O\cap \kappa$, where each $O$ is a disjoint union of open intervals with rational end-points in $(0, 1)$ such that $(0, 1)\setminus O$ has a nonempty interior in $(0, 1)$. For each set $\bigcup_{i=1}^{n}O_{i}\in \mathcal{F}$ and a finite collection $\{A_{1}, A_{2}\, \ldots, A_{n}\}$, where each $A_{i}\in\{\{g_{j}\}: j=1, \ldots, N\}$, put $$W(\bigcup_{i=1}^{n}O_{i}, \{A_{1}, A_{2}\, \ldots, A_{n}\})=\{x\in G^{\kappa}: x(\alpha)\in A_{i}\ \mbox{for each}\ \alpha\in O_{i}, i\leq N\}.$$It is easy to check that the family $\mathcal{W}$ consisting of all such sets $W(\bigcup_{i=1}^{n}O_{i}, \{A_{1}, A_{2}\, \ldots, A_{n}\})$ is  a countable $\pi$-network for $G^{\kappa}$. Clearly, each $W(\bigcup_{i=1}^{n}O_{i}, \{A_{1}, A_{2}\, \ldots, A_{n}\})$ contains a copy of $G^{\omega}$ since $(0, 1)\setminus \overline{\bigcup_{i=1}^{n}O_{i}}$ has a nonempty interior in $(0, 1)$.
\end{proof}

\begin{lemma}\label{2022p1}
Let $G$ be the quaternion group with discrete topology. Then $G^{\kappa}$ is $d$-independent for any $\omega\leq\kappa\leq\mathfrak{c}$.
\end{lemma}

\begin{proof}
Take any subgroup $H$ of $G^{\kappa}$ such that $|H|<\mathfrak{c}$. Obviously, we have $\langle x\rangle\cap H=\{e\}$ for each $x\in G^{\kappa}\setminus H$ with $x^{2}\not\in H$. By Lemma~\ref{2022l0}, there is a countable $\pi$-network $\mathcal{W}=\{W_{n}: n\in\omega\}$ for $G^{\kappa}$ such that $W_{n}$ contains a copy of $G^{\omega}$ for each $n\in\omega$. Pick any element $w_{0}\in W_{0}\setminus H$ with $w_{0}^{2}\not\in H$. This is possible since $|H|<\mathfrak{c}$ and $W_{0}$ contains a copy of $G^{\omega}$. Then $\langle w_{0}\rangle\cap H=\{e\}.$ Similarly, pick an any element $w_{1}\in W_{1}\setminus \langle H\cup\{w_{0}\}\rangle$ such that $w_{1}^{2}\not\in \langle H\cup\{w_{0}\}\rangle$. Again, this is possible since $|\langle H\cup\{w_{0}\}\rangle|<\mathfrak{c}$ and $W_{1}$ contains a copy of $G^{\omega}$. Assume that elements $w_{0}, w_{1}, \cdots, w_{n-1}$ of $G^{\kappa}$ has been defined. Then we can pick an element $w_{n}\in W_{n}\setminus\langle H\cup\{w_{0}, w_{1}, \cdots, w_{n-1}\}\rangle$ with $w_{n}^{2}\not\in\langle H\cup\{w_{0}, w_{1}, \cdots, w_{n-1}\}\rangle$, which implies that $\langle w_{0}, w_{1}, \cdots, w_{n}\rangle\cap H=\{e\}$.

Now put $D=\{w_{n}: n\in\omega\}$ and let $S$ be the subgroup of $G^{\kappa}$ generated by $D$. Since each $w_{n}\in W_{n}$, it follows that $S$ is dense in $G$. Moreover, it is obvious that $S\cap H=\{e\}$. Therefore, $G^{\kappa}$ is $d$-independent for any $\omega\leq\kappa\leq\mathfrak{c}$.
\end{proof}

By a similar proof of \cite[Proposition 3.1]{RT2021}, we have the following lemma.

\begin{lemma}\label{2022l2}
Let $\{G_{\alpha}: \alpha\in\kappa\}$ be a family of $d$-independent semitopological groups, where $1\leq\kappa\leq\mathfrak{c}$. Then the semitopological group $G=\Pi_{\alpha\in\kappa}G_{\alpha}$ is $d$-independent.
\end{lemma}

\begin{example}\label{2022e0}
There exists a $d$-independent paratopological abelian group $G$ which is not a topological group.
\end{example}

\begin{proof}
Let $(G, \tau)$ be the Sorgenfrey line. Then $(G, \tau)$ is an almost topological group which is not a topological group. Clearly, the finest Hausdorff topological group topology $\gamma$ which is coarser than $\tau$ is the usual topology on $\mathbb{R}$. By \cite[Theorem 5.2]{RT2021}, $(G, \gamma)$ is $d$-independent. Then it follows from Proposition~\ref{2022p0} that $(G, \tau)$ is $d$-independent.
\end{proof}

\begin{example}\label{2022e1}
There exists a $d$-independent paratopological non-abelian group $G$ which is not a topological group.
\end{example}

\begin{proof}
Let $G$ be the Sorgenfrey line in Example~\ref{2022e0} and $H$ be the quaternion group with discrete topology. Then $G\times H^{\kappa}$ is a non-abelian paratopological group which is not a topological group, where $\omega\leq\kappa\leq\mathfrak{c}$. From Lemma~\ref{2022p1}, Lemma~\ref{2022l2}and the proof of Example~\ref{2022e0}, it follow that $G\times H^{\kappa}$ is $d$-independent.
\end{proof}

\begin{proposition}\label{2022p3}
A cosmic semitopological group $G$ with $|G|=\mathfrak{c}$ is $d$-independent iff $G$ is maximally fragmentable.
\end{proposition}

\begin{proof}
The proof of the necessity is similar to the proof of \cite[Proposition 2.3]{RT2021}, hence it suffices to prove the sufficiency.
Assume that $G$ is maximally fragmentable. Then there exists an independent family $\{D_{\alpha}: \alpha<\mathfrak{c}\}$ of dense subgroup of $G$. Since $G$ is cosmic, each $D_{\alpha}$ is separable, hence there exists a countable subset $S_{\alpha}$ of $D_{\alpha}$ such that $D_{\alpha}\subset \overline{S_{\alpha}}$ for each $\alpha<\mathfrak{c}$. Then each $\langle S_{\alpha}\rangle$ is dense subgroup of $D_{\alpha}$, thus it is dense in $G$. Obviously, $\{\langle S_{\alpha}\rangle: \alpha<\mathfrak{c}\}$ is an independent family of dense subgroups of $G$. Take any subgroup $H$ of $G$ with $|H|<\mathfrak{c}$ and let $H_{\alpha}=H\cap \langle S_{\alpha}\rangle$ for each $\alpha<\mathfrak{c}$. Obviously, $\{H_{\alpha}: \alpha<\mathfrak{c}\}$ is an independent family of subgroups of $H$. From $|H|<\mathfrak{c}$, it follows that $|H_{\alpha}|=1$ for some $\alpha<\mathfrak{c}$, that is, $H_{\alpha}=\{e\}$. Therefore, $H\cap \langle S_{\alpha}\rangle=\{e\},$ which shows that $G$ is $d$-independent.
\end{proof}

\begin{corollary}\cite[Proposition 2.3]{RT2021}
A second countable topological abelian group $G$ with $|G|=\mathfrak{c}$ is $d$-independent iff $G$ is maximally fragmentable.
\end{corollary}

\begin{remark}
(1) The free abelian topological group $A(\mathbb{R})$ is a cosmic space with $|A(\mathbb{R})|=\mathfrak{c}$. Clearly,  $r_{0}(A(\mathbb{R}))\geq\mathfrak{c}$. By \cite[Corollary 3.6]{RT2021}, $A(\mathbb{R})$ is $d$-independent, thus it is maximally fragmentable by Proposition~\ref{2022p3}. However, $A(\mathbb{R})$ is not second countable.

(2) Let $G$ be the Sorgenfrey line. From the proof of \cite[Theorem 4.4]{LLL2015}, it follows that $G^{\omega}$ is a separable and maximally fragmentable paratopological group satisfying $|G|=\mathfrak{c}$.  By Proposition~\ref{2022p0}, $G^{\omega}$ is $d$-independent. However, $G^{\omega}$ is not a cosmic space and non-precompact. Moreover, let $H$ be unit plane with the Sorgenfrey topology. Then $H^{\omega}$ is precompact and is not a cosmic space. Since $H^{\omega}$ endowed with usual topology is $d$-independent by \cite[Corollary 4.3]{RT2021}, it follows from Proposition~\ref{2022p0} that $H^{\omega}$ is $d$-independent. Then $H^{\omega}$ is maximally fragmentable separable by \cite[Theorem 4.1]{RT2021}. Therefore, the following question is interesting.
\end{remark}

\begin{question}
Is each maximally fragmentable separable semitopological group $G$ satisfying $|G|=\mathfrak{c}$ $d$-independent? What if, additionally, $G$ is precompact?
\end{question}

Now we consider dense-subgroup-separable paratopological groups. The following lemma is provided in \cite[Lemma 2.6]{RT2021} for the case of abelian.

\begin{lemma}\label{2022l3}
Let $G$ be a group. If $S$ is a subgroup of $G$ with $|S|<\mathfrak{c}$ and $A$ is a subset of $G$ with $|A|\geq\mathfrak{c}$ such that $\langle x\rangle\cap \langle y\rangle=\{e\}$ for distinct $x, y\in A$, then there exists $x\in A$ such that $\langle x\rangle\cap S=\{e\}.$
\end{lemma}

\begin{theorem}\label{2022t11}
 Let $G$ be a regular dense-subgroup-separable abelian semitopological group. Then the following conditions are equivalent:
 \begin{enumerate}
\item $G$ is $d$-independent;

 \item for each subgroup $S$ of $G$ with $|S|<\mathfrak{c}$ and each nonempty open set $U$ in $G$, there is $x\in U\setminus\{e\}$ such that $\langle x\rangle\cap S=\{e\};$

\item every nonempty open set in $G$ contains an independent subset of cardinality $\mathfrak{c}$.
 \end{enumerate}
\end{theorem}

\begin{proof}
(1) $\Rightarrow$ (2). Assume that $G$ is $d$-independent. Take any subgroup $S$ of $G$ with $|S|<\mathfrak{c}$ and any nonempty open set $U$ in $G$. Since $G$ is $d$-independent, there exists a countable dense subgroup $H$ such that $H\cap S=\{e\}$. Since any $d$-independent semitopological group is not discrete, it follows that $U\cap (H\setminus\{e\})\neq\emptyset$. Pick any $x\in U\cap (H\setminus\{e\})$. Then $\langle x\rangle\subset H$, thus $\langle x\rangle\cap S=\{e\}$.

(2) $\Rightarrow$ (1). Take any subgroup $S$ of $G$ with $|S|<\mathfrak{c}$. From the separability and regularity of $G$, it follows from \cite[Thoerem 1.57]{E1989} that $w(G)\leq 2^{\omega}=\mathfrak{c}$. Assume that $\mathcal{B}=\{U_{\alpha}: \alpha<\mathfrak{c}\}$ is a base for $G$. By induction, we can define a subset $D=\{x_{\alpha}: \alpha\in\mathfrak{c}\}$ of $G$ such that $x_{\alpha}\in U_{\alpha}$ and $\langle x_{\alpha}\rangle\cap S_{\alpha}=\{e\}$ for every $\alpha<\mathfrak{c}$, where $S_{0}=S$ and $S_{\alpha}=\langle S\cup\{x_{\beta}: \beta<\alpha\}\rangle$. Clearly, each $|S_{\alpha}|\leq |S|\cdot \alpha\cdot\omega<\mathfrak{c}$ and $D$ is dense in $G$. Put $H=\langle D\rangle$. It is easily verified that $H\cap S=\{e\}$. Since $G$ is dense-subgroup-separable and $H$ is dense in $G$, it follows that $H$ contains a countable dense subgroup $A$ in $H$. Hence $A$ is dense in $G$ and $A\cap S\subset H\cap S=\{e\}$. Therefore, $G$ is $d$-independent.

(2) $\Rightarrow$ (3). Let $A$ be a maximal independent subset of a nonempty open set $U$ in $G$. We claim that $|A|\geq \mathfrak{c}$. Suppose not, then $|A|<\mathfrak{c}$. Put $S=\langle A\rangle$. Clearly, $|S|\leq \mathfrak{c}$. By our assumption, there exists an element $z\in U\setminus\{e\}$ so that
$\langle z\rangle\cap S=\{e\}$. Hence $B=A\cup \{z\}$ is an independent subset of $U$, which contradicts the maximality of the set $A$.

(3) $\Rightarrow$ (2). The implication is immediate from Lemma~\ref{2022l3}.
\end{proof}

\begin{corollary}
Let $G$ be a regular separable abelian semitopological group with a countable tightness. Then $G$ is $d$-independent iff for each subgroup $S$ of $G$ with $|S|<\mathfrak{c}$ and each nonempty open set $U$ in $G$, there is $x\in U\setminus\{e\}$ such that $\langle x\rangle\cap S=\{e\}.$
\end{corollary}

\begin{theorem}\label{2022t22}
If $G$ is a regular dense-subgroup-separable abelian semitopological group with $r_{0}(G)\geq\mathfrak{c}$, then $G$ is  $d$-independent.
\end{theorem}

\begin{proof}
By Theorem~\ref{2022t11}, it suffices to prove that for each subgroup $S$ of $G$ with $|S|<\mathfrak{c}$ and each nonempty open set $U$ in $G$, there is $x\in U\setminus\{e\}$ such that $\langle x\rangle\cap S=\{e\}.$

Take any subgroup $S$ of $G$ with $|S|<\mathfrak{c}$ and nonempty open set $U$ in $G$. Since $G$ is dense-subgroup-separable, it follows that $G$ is separable, then there exists a countable subgroup $C$ of $G$ such that $U+C=G$. Moreover, it follows from$r_{0}(G)\geq\mathfrak{c}$ that $G$ has an independent set $D$ of elements of infinite order such that $|D|=\mathfrak{c}$. Hence there exists $c\in C$ such that $|D\cap (U+c)|=\mathfrak{c}$. Let $D^{\ast}=D\cap (U+c)$. Put $R=S+C$. Clearly, $|R|<\mathfrak{c}$. From Lemma~\ref{2022l3}, it follows that there exists $d\in D^{\ast}\setminus\{c\}$ such that $\langle d\rangle\cap T=\{e\}$. Put $y=d-c$. Clearly, $y\neq e$ and $y$ belongs to $U$. We claim that $\langle y\rangle\cap S=\{e\}$.

Indeed, suppose not, there exists positively natural number $n$ such that $ny\in S$, that is, $ny=nd-nc\in S$. Since $c\in T$ and $ny\in T$, it follows that $nd\in T$. However, $\langle d\rangle\cap T=\{e\}$, which is a contradiction.
\end{proof}

\begin{corollary}\label{2022cc}
If $G$ is a regular dense-subgroup-separable abelian semitopological group with $|G|\geq\mathfrak{c}$ and $|tor(G)|<\mathfrak{c}$, then $G$ is  $d$-independent.
\end{corollary}

\begin{proof}
Since $|G|\geq\mathfrak{c}$ and $|tor(G)|<\mathfrak{c}$, it follows that $|G/tor(G)|\geq\mathfrak{c}$. Then $r_{0}(G)\geq\mathfrak{c}$ by \cite[Remark 2.16]{RT2021}. Therefore,  the required conclusion follows from Theorem~\ref{2022t22}.
\end{proof}

\begin{lemma}\label{2022l4}
Let $G$ be a dense-subgroup-separable paratopological abelian group with $|G|>1$. If $G$ is bounded, every primary component of $G$ is dense-subgroup-separable.
\end{lemma}

\begin{proof}
Let $G=\bigoplus_{1\leq k\leq n}G_{p_{k}}$ be the decomposition of $G$ into the direct sum of its primary components. Without loss of generality, we only prove $G_{p_{n}}$ is dense-subgroup-separable since the proof of the other primary component is similar. If $n=1$, then it is obvious. Hence we may assume that $n>1$. Let $\psi$ be the natural homomorphism of $G$ onto $G_{p_{n}}$, $\psi(x)=p_{1}\cdots p_{n-1}x$ for each $x\in G$. Clearly, $\psi$ is open and continuous. Take any dense subgroup $H$ of $G_{p_{n}}$. Then $$\psi^{-1}(H)=(\bigoplus_{1\leq k\leq n-1}G_{p_{k}})\bigoplus H$$ is dense subgroup in $G$, hence there exists a countable subset $D$ of $\psi^{-1}(H)$ such that $D$ is dense in $\psi^{-1}(H)$. Then $\psi(D)$ is a countable set and dense in $H$. Therefore, $G_{p_{n}}$ is dense-subgroup-separable.
\end{proof}

\begin{theorem}\label{2022t23}
Let $G$ be a regular dense-subgroup-separable paratopological abelian group with $|G|>1$. If $G$ is bounded, then the following conditions are equivalent:
 \begin{enumerate}
\item $G$ is $d$-independent;

 \item $G$ is a nontrivial $M$-group;

\item each nontrivial primary component $G_{p}$ of $G$ $d$-independent.
 \end{enumerate}
\end{theorem}

\begin{proof}
By Proposition~\ref{2022p4}, we have (1) $\Rightarrow$ (2). It suffice to prove (2) $\Rightarrow$ (3) and (3) $\Rightarrow$ (1).

(2) $\Rightarrow$ (3). Assume that $G$ is a nontrivial $M$-group, then it follows from \cite[Lemma 2.22]{RT2021} that every nontrivial primary component $G_{p}$ of $G$ is also an $M$-group. Fix any nontrivial primary component $G_{p}$ of $G$. Now we prove that $G_{p}$ is $d$-independent. Indeed, it follows from Lemma~\ref{2022l4} that $G_{p}$ is dense-subgroup-separable. By Theorem~\ref{2022t11}, it suffices to prove that every nonempty open set in $G_{p}$ contains an independent subset of cardinality $\mathfrak{c}$. Since $G$ is bounded and $G_{p}$ is a nontrivial $M$-group, there exists a positive integer $k$ such that $G_{p}$ is a subgroup of period $p^{k}$ and $|p^{k-1}G_{p}|\geq\mathfrak{c}$. Put $H=p^{k-1}G_{p}$. Then $H$ is a subgroup of a prime period $p$. Take any non-empty open subset $U$ of $G_{p}$. Since $G_{p}$ is separable, there exists a countable subset $C$ of $G_{p}$ such that $U+C=G_{p}$. Then $p^{k-1}U+p^{k-1}C=H.$ Since $|H|\geq\mathfrak{c}$ and $p^{k-1}C$ is countable, it follows that $|p^{k-1}U|\geq\mathfrak{c}$, then $p^{k-1}U$ contains an independent subset $A$ of cardinality $\mathfrak{c}$ by \cite[Lemma 2.18]{RT2021}. For each $a\in A$, take any $x(a)\in U$ such that $p^{k-1}x(a)=a$. Put $B=\{x(a): a\in A\}$. Then it is easily verified that $B$ is an independent subset of cardinality $\mathfrak{c}$.

(3) $\Rightarrow$ (1). Assume that $G=\bigoplus_{1\leq k\leq n}G_{p_{k}}$ is the decomposition of $G$ into the direct sum of its primary components, where each primary component $G_{p_{k}}$ of $G$ is $d$-independent. Take an arbitrary subgroup $S$ of $G$ with $|S|<\mathfrak{c}$. From Lemma~\ref{2022l2}, the direct product $\Pi_{1\leq k\leq n}G_{p_{k}}$ is a $d$-independent paratopological group, where each factor $G_{p_{k}}$ inherits the topology from the group $G$. Let $\phi$ be the natural homomorphism of $\Pi_{1\leq k\leq n}G_{p_{k}}$ onto $G$, $\phi(x_{1}, \ldots, x_{n})=x_{1}+ \ldots+ x_{n}$; then $\phi$ is a continuous bijection since $G$ is a paratopological group. Hence $\phi^{-1}(S)$ is a subgroup of $\Pi_{1\leq k\leq n}G_{p_{k}}$ with $|\phi^{-1}(S)|<\mathfrak{c}$. Since $\Pi_{1\leq k\leq n}G_{p_{k}}$ is $d$-independent, there exists a countable dense subgroup $D$ of $\Pi_{1\leq k\leq n}G_{p_{k}}$  such that $D\cap \phi^{-1}(S)=\{e\}$. Then $\phi(D)$ is dense in $G$ and $\phi(D)\cap S=\{e\}$. Therefore, $G$ is $d$-independent.
\end{proof}

\begin{remark}
There exists a compact bounded separable metrizable (thus dense-subgroup-separable) topological abelian group $G$ such that $G$ is not $d$-independent. Indeed, let $G=\mathbb{Z}(p)^{\omega}\times \mathbb{Z}(p)^{\omega}\times \mathbb{Z}(p^{2})$, where $\mathbb{Z}(p)$ and $\mathbb{Z}(p^{2})$ are endowed with discrete topologies. By \cite[Example 2.23]{RT2021}, $G$ is not an $M$-group, hence it is not $d$-independent by Theorem~\ref{2022t23}.
\end{remark}

\begin{theorem}
Let $G$ be a regular torsion paratopological abelian $p$-group for some prime $p$ such that each bounded subgroup is dense-subgroup-separable. Then $G$ is $d$-independent iff it is a nontrivial $M$-group.
\end{theorem}

\begin{proof}
The necessity is obvious by Proposition~\ref{2022p4}. Now assume that $G$ is a nontrivial $M$-group; thus $|G|\geq\mathfrak{c}$. If $G$ is bounded, then $G$ is $d$-independent by Theorem~\ref{2022t23}. Suppose that $G$ is unbounded. For each positive integer $k$, let $G_{k}=G[p^{k}]=\{x\in G: p^{k}x=e\}$. Clearly, we have $G_{1}\subset G_{2}\subset\cdots$ and $G=\bigcup_{k\geq 1}G_{k}$. For each non-negative integer $k$, let $H_{k}=p^{k}G$. Since $G$ is unbounded $M$-group, it follows that $|H_{k}|\geq\mathfrak{c}$, then $H_{k}[p]$ of $H_{k}$ has cardinality at least $\mathfrak{c}$ by \cite[Lemma 9.9.15]{AT2008}. We claim that each $G_{k}$ is an $M$-group. Indeed, fix any positive integer $k$. Take any $n\in\mathbb{N}$. Then there exists a biggest integer $m$ such that $p^{m}$ divides $n$. If $m\geq k$, then $nG_{k}=\{e\}$; otherwise, assume that $m<k$, then it is easily verified that $H_{m}[p]\subset p^{m}G_{k}=nG_{k}$, which shows that $|nG_{k}|\geq\mathfrak{c}$ since $|H_{m}[p]|\geq\mathfrak{c}$. Therefore, each $G_{k}$ is an $M$-group. Next we prove that $G$ is $d$-independent.

Take any subgroup $S$ of $G$ with $|S|<\mathfrak{c}$. By induction, we can define a sequence of countable dense subgroups $\{D_{n}: n\in\omega\}$ such that the following conditions are satisfied:

\smallskip
(1) $D_{0}=\{e\}$;

\smallskip
(2) each $D_{k}$ is a dense subgroup of $G_{k}$;

\smallskip
(3) $D_{k}\cap (S+D_{1}+\cdots+D_{k-1})=\{e\}$.

Now put $D=\bigcup_{n\geq 1}(D_{1}+\cdots +D_{n})$. Then $D$ is a countable dense subgroup of $G$ and $D\cap S=\{e\}$. Therefore, $G$ is $d$-independent.
\end{proof}

By a similar proof of \cite[Theorem 2.26]{RT2021}, we have the following result.

\begin{theorem}\label{2022t33}
Let $G$ be a semitopological abelian group with a countable $\pi$-network $\mathcal{N}$. If each element $N\in\mathcal{N}$ contains an independent subset of size $\mathfrak{c}$, then $G$ is $d$-independent.
\end{theorem}

\begin{lemma}\label{2022l44}
Let $G$ be a paratopological group with a countable network. If $|G|=\kappa$ and $cf(\kappa)>\omega$, then $G$ has a countable network $\mathcal{N}$ such that $|\mathcal{N}|=\kappa$ for each $N\in\mathcal{N}$.
\end{lemma}

\begin{proof}
Let $\mathcal{D}$ be a countable network for $G$. Put $\mathcal{D}^{\ast}=\{D\in \mathcal{D}: |D|=\kappa\}$. Since $|G|=\kappa$ and $cf(\kappa)>\omega$, it follows that $\mathcal{D}^{\ast}\neq\emptyset$ and $|\bigcup(\mathcal{D}\setminus\mathcal{D}^{\ast})|<\kappa$. We claim that $|V|=\kappa$ for each open neighborhood $V$ of $e$. Suppose not, we can assume that $G$ a base $\mathcal{B}$ of $G$ such that $|B|<\kappa$ for each element $B$ of $\mathcal{B}$. Then it is easy to see that $G$ has a countable network $\mathcal{N}$ such that $|B|<\kappa$ for each $B\in\mathcal{N}$, which is a contradiction with $|G|=\kappa$. Therefore, we can find a base $\mathcal{B}$ of $G$ such that $|B|=\kappa$ for each element $B$ of $\mathcal{B}$. Let $\mathcal{H}=\{D_{1}D_{2}: D_{1}\in\mathcal{D}, D_{2}\in\mathcal{D}^{\ast}\}$. Clearly, $\mathcal{H}$ is countable. Next we prove that $\mathcal{H}$ is a network for $G$.

Indeed, take an arbitrary element $x\in G$ and an open neighborhood $U$ of $x$ in $G$. Then there exists an open neighborhood $V$ of $e$ in $G$ such that $xV^{2}\subset U$. Hence there exists $C_{1}\in\mathcal{D}$ such that $x\in C_{1}\subset xV$. Moreover, it is obvious that $|V|=\kappa$, hence $V\setminus(\bigcup(\mathcal{D}\setminus\mathcal{D}^{\ast}))\neq\emptyset$. Take any $y\in V\setminus(\bigcup(\mathcal{D}\setminus\mathcal{D}^{\ast}))$; then there exists $C_{2}\in \mathcal{D}^{\ast}$ such that $y\in C_{2}\subset V$. Therefore, we have $x=xe\in C_{1}C_{2}\subset xVV\subset U$. Therefore, $\mathcal{H}$ is a countable network for $G$.
\end{proof}

\begin{theorem}
Let $G$ be an almost torsion-free paratopological abelian group with a countable network. If $|G|=\mathfrak{c}$, then $G$ is $d$-independent.
\end{theorem}

\begin{proof}
By Lemma~\ref{2022l44}, $G$ has a countable network $\mathcal{N}$ such that $|N|=\mathfrak{c}$ for every $N\in\mathcal{N}$. Then the required conclusion follow from Corollary~\ref{2022cc} and Theorem~\ref{2022t33}.
\end{proof}

\begin{corollary}\label{2022c2024}
A separable metrizable almost torsion-free paratopological abelian group $G$ with $|G|=\mathfrak{c}$ is $d$-independent.
\end{corollary}

We say that a topological group has {\it no small subgroups} if there is a neighborhood $U$ of the neutral element $e$ such that for each subgroup $H$ of $G$ the relation $H\subset U$ implies $H=\{e\}$. A compact group $G$ is called a {\it compact Lie group} if it has no small subgroups.

\begin{corollary}
If an infinite compact abelian group $G$ is a compact Lie group, then $G$ is $d$-independent.
\end{corollary}

\begin{proof}
From \cite[Proposition 2.42]{HM1998}, $G$ is isomorphic to $\mathbb{T}^{n}\times E$ for some positive integer and finite abelian group. Clearly, $G$ is separable metrizable, almost torsion-free and $|G|=\mathfrak{c}$. By Corollary~\ref{2022c2024}, $G$ is $d$-independent.
\end{proof}

\begin{problem}\cite[Problem 2.35]{RT2021}
Can one drop ``abelian'' in Corollary~2.34? What if $G$ is compact?
\end{problem}

\begin{problem}\cite[Problem 2.36]{RT2021}
Let $G$ be a nontrivial compact connected group with $w(G)\leq\mathfrak{c}$. Must $G$ be $d$-independent?
\end{problem}

\begin{problem}\cite[Problem 2.39]{RT2021}
Is every connected separable metrizable topological abelian group $G$ with $|G|>1$ $d$-independent?
\end{problem}

We say that a topological group $G$ is {\it MAP} iff the family of continuous homomorphisms to unitary
groups separate its points. By \cite{DMT2005}, the topological abelian groups that are either locally compact or precompact is MAP.

\begin{theorem}
Let $G$ be a dense-subgroup-separable MAP abelian group with a nontrivial connected component. Then $G$ is $d$-independent.
\end{theorem}

\begin{proof}
By Theorem~\ref{2022t22}, it suffices to prove that $r_{0}(G)\geq\mathfrak{c}$. Let $C$ be the connected component of $G$; then $|C|>1$ and $C$ is a closed subgroup of $G$.
Pick any $b\in C\setminus\{e\}$. Since $G$ is MAP, it follows from \cite[9.6.a]{AT2008} that there exists a continuous homomorphism $f$ from $G$ to $\mathbb{T}^{\omega}$ such that $f(b)\neq e$, where $e$ is the identity element of $\mathbb{T}^{\omega}$. Then there exists $n_{0}\in\omega$ such that $b_{n_{0}}\neq 1$, where $f(b)=(b_{i})_{i\in\omega}$ and $b_{i}\in\mathbb{T}$ for each $i\in\omega$. Let $\pi_{n_{0}}: \mathbb{T}^{\omega}\rightarrow\mathbb{T}$ be the naturally projective mapping, that is, $\pi_{n_{0}}(x)=\pi_{n_{0}}((x_{i})_{i\in\omega})=x_{n_{0}}$ for each $x\in\mathbb{T}^{\omega}$. Then $\pi_{n_{0}}\circ f(C)$ is a non-trivial connected subgroup of $\mathbb{T}$, hence $\pi_{n_{0}}\circ f(C)=\mathbb{T}$. Therefore, $r_{0}(\pi_{n_{0}}\circ f(C))\geq\mathfrak{c}$. We claim that $r_{0}(G)\geq\mathfrak{c}$. Indeed, let $D$ be an independent subset of $\pi_{n_{0}}\circ f(C)$ with $|D|=\mathfrak{c}$. For each $d\in D$, take any point $g_{d}\in f^{-1}(\pi_{n_{0}}^{-1}(d))$. Put $D^{\prime}=\{g_{d}: d\in D\}$; then $D^{\prime}$ is an independent subset of $G$ with $|D^{\prime}|=\mathfrak{c}$, thus $r_{0}(G)\geq\mathfrak{c}$.
\end{proof}

\begin{corollary}\label{2022ccc}
Let $G$ be a cosmic MAP abelian group with a nontrivial connected component. Then $G$ is $d$-independent.
\end{corollary}

\begin{corollary}\cite[Corollary 2.33]{RT2021}
Let $G$ be a second-countable locally compact abelian group with a nontrivial connected component. Then $G$ is $d$-independent.
\end{corollary}

\begin{corollary}
Each subgroup with a nontrivial connected component of a cosmic precompact abelian group is $d$-independent.
\end{corollary}

\begin{proof}
Let $G$ be cosmic precompact abelian group, and let $H$ be a subgroup of $G$ such that $H$ has a nontrivial connected component. Clearly, $H$ is cosmic. Then it follows from \cite[Proposition 3.7.4]{AT2008} that $H$ is precompact, thus $H$ is MAP. Then $H$ is $d$-independent by Corollary~\ref{2022ccc}.
\end{proof}

\begin{remark}
It is well known that under the assumption of set-theory there exists a pseudocompact, connected, hereditarily separable, non-metrizable topological abelian group $G$ (see \cite{DD2005} or \cite{HJ1976}), thus $G$ is dense-separable which shows that $G$ is $d$-independent; however, $G$ is not locally compact.
\end{remark}

By a similar proof of \cite[Lemma 2.37]{RT2021}, we have the following lemma. For the convenience of the authors, we give out the proof.

\begin{lemma}\label{2022lll}
Suppose $G$ is a nontrivially pseudocompact connected abelian group such that $|G^{\wedge}|\leq w(G)\leq\mathfrak{c}$, where $G^{\wedge}$ is the dual group of $G$ which is endowed with the pointwise convergence topology on elements of $G$. Then $G$ has an independent set $\{d_{\alpha}: \alpha<\mathfrak{c}\}$ such that the cyclic subgroup $\langle d_{\alpha}\rangle$ is dense in $G$, for each $\alpha<\mathfrak{c}$.
\end{lemma}

\begin{proof}
From \cite[Corollary 2.2]{BT2012}, it follows that the canonical evaluation mapping $\alpha_{G}: G\rightarrow (G^{\wedge})^{\wedge}$
is a topological isomorphism. Since $G$ is a nontrivially connected pseudocompact abelian group, it follows from \cite[9.6.g and 9.6.h]{AT2008} that $G^{\wedge}$ is infinite and torsion-free, then $G^{\wedge}$ is a subgroup of a divisible abelian group $D$ by \cite[4.1.6]{R2012}. Put $$K=\bigcap\{S\subset D: S\ \mbox{is a divisible subgroup and}\ G^{\wedge}\subset S\}.$$Clearly, $G^{\wedge}$ is an essential subgroup of $K$ and $|G^{\wedge}|=|K|$, hence $K$ is torsion-free. Since $|G^{\wedge}|\leq w(G)\leq\mathfrak{c}$, we have $\omega\leq|K|\leq\mathfrak{c}$.
Therefore, it follows from \cite[4.1.5]{R2012} that $K$ is a direct sum of isomorphic
copies of $\mathbb{Q}$ and of quasicyclic groups. Since $K$ is torsion-free, we have $K\cong \mathbb{Q}^{(\kappa)}$, where $\kappa\leq\mathfrak{c}$, then $\mathbb{Q}^{(\mathfrak{c})}\cong \mathbb{Q}^{(\kappa\times \mathfrak{c})}\cong K^{(\mathfrak{c})}=\bigoplus_{\alpha<\mathfrak{c}}K_{\alpha}$, where every $K_{\alpha}$ is an isomorphic copy of $K$. It is easy to see that $\mathbb{Q}^{(\mathfrak{c})}$ is algebraically isomorphic to a subgroup of $\mathbb{T}$, hence we can identify $\bigoplus_{\alpha<\mathfrak{c}}K_{\alpha}$ with a subgroup of $\mathbb{T}$. For every $\alpha<\mathfrak{c}$, let $\pi_{\alpha}$ be an isomorphic embedding of $G^{\wedge}$ into $K_{\alpha}$; then we can consider $\pi_{\alpha}$ as a character of $G^{\wedge}$. Since $G\cong (G^{\wedge})^{\wedge}$, for every $\alpha<\mathfrak{c}$ there exists $d_{\alpha}\in G$ such that $x(d_{\alpha})=\pi_{\alpha}(x)$ for each $x\in G^{\wedge}$. We conclude that $B=\{d_{\alpha}: \alpha<\mathfrak{c}\}$ is independent.

Indeed, let $d=n_{1}d_{\alpha_{1}}+\cdots+n_{k}d_{\alpha_{k}}$, where $n_{1}, \ldots, n_{k}\in \mathbb{Z}\setminus\{0\}$ and $\alpha_{1}, \ldots, \alpha_{k}<\mathfrak{c}$. For any $x\in G^{\wedge}\setminus\{e_{G^{\wedge}}\}$, we have $x(d)=\pi_{\alpha_{1}}(x)^{n_{1}}\cdots \pi_{\alpha_{k}}(x)^{n_{k}}\neq 1$ since the elements $\pi_{\alpha_{1}}(x), \ldots, \pi_{\alpha_{k}}(x)$ of the group $\bigoplus_{\alpha<\mathfrak{c}}K_{\alpha}$ have infinite order and are independent.

Finally we prove that each cyclic group $\langle d_{\alpha}\rangle$ is dense in $G$. Suppose not, it follows from \cite[9.6.g]{AT2008} that there exists a nontrivial character $x\in G^{\wedge}$ such that $x(d_{\alpha})=1$. However, since $\pi_{\alpha}$ is injective and $x\neq e_{G^{\wedge}}$, we conclude that $x(d_{\alpha})=\pi_{\alpha}(x)\neq 1$, which is a contradiction.
\end{proof}

By Lemmas~\ref{2022l3} and~\ref{2022lll}, we have the following theorem.

\begin{theorem}\label{2022t5656}
Each nontrivial pseudocompact connected abelian group $G$ with $|G^{\wedge}|\leq w(G)\leq\mathfrak{c}$ is $d$-independent.
\end{theorem}

Since the weight of an infinite compact abelian group $G$ coincides with the cardinality of $G^{\wedge}$, it follows from Theorem~\ref{2022t5656} that the following corollary holds.

\begin{corollary}\cite[Theorem 2.38]{RT2021}
Each nontrivial compact connected abelian group $G$ with $w(G)\leq\mathfrak{c}$ is $d$-independent.
\end{corollary}

\begin{corollary}
Let $G$ be a nontrivial divisible pseudocompact abelian group with $|G^{\wedge}|\leq w(G)\leq\mathfrak{c}$. Then $G$ is $d$-independent.
\end{corollary}

\begin{proof}
By \cite[9.11.A]{AT2008}, $G$ is connected. Then $G$ is $d$-independent by Theorem~\ref{2022t5656}.
\end{proof}

We do not know if we can drop ``$|G^{\wedge}|\leq w(G)$'' in Theorem~\ref{2022t5656}. The answer is positive if the following two questions are affirmative.

\begin{question}
Assume that $G$ is a precompact abelian group and $H$ a closed subgroup of $G^{\wedge}$ separating the points of $G$. Is $H=G^{\wedge}$?
\end{question}

\begin{question}
Does the weight of an infinite pseudocompact abelian group $G$ coincide with the cardinality of $G^{\wedge}$?
\end{question}

  %%%%%%%%%%%%%%%%%%%%%%%%%%%%

  \end{document}